\documentclass{amsart}
\usepackage{amsmath}
\usepackage{amsfonts}

\newtheorem{theorem}{Theorem}
\theoremstyle{plain}

\newtheorem{definition}{Definition}

\numberwithin{equation}{section}
\input{tcilatex}

\begin{document}
\title{On Approximation Properties for Non-linear Integral Operators}
\author{Sevgi ESEN ALMALI}
\address{Science and Art Faculty, Department of Mathematics, Kirikkale
University, Kirikkale-Turkey}
\email{sevgi\_esen@hotmail.com}
\urladdr{http://www.kku.edu.tr}
\thanks{The auhor wish to thank Prof. Dr. A. D. Gadjiev for the useful
suggestions in the matter.}

\begin{abstract}
We investigate the problem of pointwise convergence of the family of
non-linear integral operators%
\begin{equation*}
L_{\lambda
}(f,x)=\dint\limits_{a}^{b}\dsum\limits_{m=1}^{N}f^{m}(t)K_{\lambda
,m}(x,t)dt,
\end{equation*}%
where $\lambda $ is a real parameters, $K_{\lambda ,m}(x,t)$ is non-negative
kernel and $f$ is the function in $L_{1}(a,b)$. We consider two cases where $%
(a,b)$ is a finite interval and when is the whole real axis..
\end{abstract}

\maketitle

\textbf{1.\bigskip Introduction}

In $\left[ 6\right] $ the concept of singularity was extended to cover the
case of nonlinear integral operators,%
\begin{equation*}
T_{w}f(s)=\int_{G}K_{w}(t-s,f(t))dt,\text{ \ s}\in G
\end{equation*}%
the assumption of linearity of the operators being replaced by an assumption
of a Lipschitz condition for $K_{w}$ with respect to the second variable.
Recently, Swiderski and Wachnicki $\left[ 7\right] $ investigated the
pointwise convergence of the operators $T_{w}f$ in $L_{p}(-\pi ,\pi )$ and $%
L_{p}(R)$ at the points, where $x_{0}$ is a point of continuity and a
Lebesque point of $f.$ \bigskip

In $\left[ 2\right] ,$ Karsli studied both the pointwise convergence and the
rate of pointwise convergence of above operators on a $\mu -generelized$ \
Lebesque point to $f\in L_{1}(a,b)$ as $(x,\lambda )\rightarrow
(x_{0},\lambda _{0}).$In $\left[ 1\right] ,$ it is studied the rate of
convergence at a point $x,$ which has a discontinuity of \ the first kinds
as $\lambda \rightarrow \lambda _{0}.$In $\left[ 4\right] $ they obtained
convergence results and rate of approximation for functions belonging to
BV-spaces by mean of nonlinear convolution integral operators.

The aim of the article is to obtain pointwise convergence results for a
family of non-linear operators of the form%
\begin{equation}
L_{\lambda
}(f,x)=\dsum\limits_{m=1}^{N}\dint\limits_{a}^{b}f^{m}(t)K_{\lambda
,m}(x,t)dt  \tag{1}
\end{equation}%
where $K_{\lambda ,m}(x,t)$ is a family of kernels depending on $\lambda .$%
We study convergence of the family (1) at every Lebesque point of the
function $f$ in the spaces of $L_{1}(a,b)$ and $L_{1}(-\infty ,\infty ).$

Now we give the following definition.

\begin{definition}
(Class A): We take a family ($K_{\lambda })_{\lambda \in \Lambda }$ of
functions $K_{\lambda ,m}(x,t):RXR\rightarrow R$. We will say that the
function $K_{\lambda }(x,t)$ belongs the class A, if the following
conditions are satisfied.
\end{definition}

a) $K_{\lambda ,m}(x,t)$ is a non-negative function defined for all $t$ on $%
(a,b)$ and $\lambda \in \Lambda .$

b) As function of $t$, $K_{\lambda ,m}(x,t)$ is non-decreasing on $\left[ a,x%
\right] $ and non-increasing on $\left[ x,b\right] $ for any fixed $x$

c)For any fixed $x,$ $\underset{\lambda \rightarrow \infty }{\lim }%
\dint\limits_{a}^{b}K_{\lambda ,m}(x,t)dt=C_{m}$.

d ) For every $m,\underset{\lambda \rightarrow \infty }{1\leq m\leq N\text{
and every }y\neq x\text{ , }\lim }K_{\lambda ,m}(x,y)=0.$\bigskip \newline

\bigskip \textbf{2.} \textbf{Main Result}

We are going to prove the family of \ non-linear integral operators (1) with
the positive kernel convergence to the functions $f\in L_{1}(a,b)$

\begin{theorem}
Suppose that $f\in L_{1}(a,b)$ and $f$ is bounded on $(a,b).$ If
non-negative the kernel $K_{\lambda ,m}$ belongs to Class A, then, for the
operator $L_{\lambda }(f,x)$ which is defined in (1) 
\begin{equation*}
\lim_{\lambda \rightarrow \infty }L_{\lambda
}(f,x_{0})=\dsum\limits_{m=1}^{N}C_{m}f^{m}(x_{0})
\end{equation*}%
holds at every $x_{0}-$lebesque point of $f$ function
\end{theorem}

\begin{proof}
For integral (1) we can write%
\begin{eqnarray*}
L_{\lambda }(f,x_{0})-\dsum\limits_{m=1}^{N}C_{m}f^{m}(x_{0})
&=&\dsum\limits_{m=1}^{N}\dint\limits_{a}^{b}f^{m}(t)-f^{m}(x_{0})K_{\lambda
,m}(x_{0},t)dt \\
&&+\dsum\limits_{m=1}^{N}f^{m}(x_{0})\left[ \dint\limits_{a}^{b}K_{\lambda
,m}(x_{0},t)dt-C_{m}\right]
\end{eqnarray*}%
and in view of a) 
\begin{eqnarray*}
\left\vert L_{\lambda
}(f,x_{0})-\dsum\limits_{m=1}^{N}C_{m}f^{m}(x_{0})\right\vert &\leq
&\dsum\limits_{m=1}^{N}\dint\limits_{a}^{b}\left\vert
f^{m}(t)-f^{m}(x_{0})\right\vert K_{\lambda ,m}(x_{0},t)dt \\
&&+\dsum\limits_{m=1}^{N}\left\vert f^{m}(x_{0})\right\vert \left\vert
\dint\limits_{a}^{b}K_{\lambda ,m}(x_{0},t)dt-C_{m}\right\vert \\
&=&I_{1}(x_{0},\lambda )+I_{2}(x_{0},\lambda ).
\end{eqnarray*}%
It is suffcient to show that terms on right hand side of the last inequality
tend to zero as $\lambda \rightarrow \infty .$By property c) ,it is clear
that $I_{2}(x_{0},\lambda )$ tends to zero as $\lambda \rightarrow \infty .$

Now we consider $I_{1}(x_{0},\lambda ).$ For any fixed $\delta >0$, we can
write $I_{1}(x_{0},\lambda )$ as follow.

\begin{eqnarray}
I_{1}(x_{0},\lambda ) &=&\dsum\limits_{m=1}^{N}\left[ \dint%
\limits_{a}^{x_{0}-\delta }+\dint\limits_{x_{0}-\delta
}^{x_{0}}+\dint\limits_{x_{0}}^{x_{0}+\delta }+\dint\limits_{x_{0}+\delta
}^{b}\right] \left\vert f^{m}(t)-f^{m}(x_{0})\right\vert K_{\lambda
,m}(x_{0},t)dt  \TCItag{2} \\
&=&I_{11}(x_{0},\lambda ,m)+I_{12}(x_{0},\lambda ,m)+I_{13}(x_{0},\lambda
,m)+I_{14}(x_{0},\lambda ,m)  \notag
\end{eqnarray}%
Firstly we shall calculate $I_{11}(x_{0},\lambda ,m)$, that's

\begin{equation*}
I_{11}(x_{0},\lambda
,m)=\dsum\limits_{m=1}^{N}\dint\limits_{a}^{x_{0}-\delta }\left\vert
f^{m}(t)-f^{m}(x_{0})\right\vert K_{\lambda ,m}(x_{0},t)dt.
\end{equation*}%
By the condition b), we have

\begin{equation}
I_{11}(x_{0},\lambda ,m)\leq \dsum\limits_{m=1}^{N}K_{\lambda
,m}(x_{0},x_{0}-\delta )\left\{ \dint\limits_{a}^{x_{0}-\delta }\left\vert
f^{m}(t)\right\vert dt+x\dint\limits_{a}^{x_{0}-\delta }\left\vert
f^{m}(x_{0})\right\vert dt\right\}  \notag
\end{equation}%
and%
\begin{equation}
\leq \dsum\limits_{m=1}^{N}K_{\lambda ,m}(x_{0},x_{0}-\delta )\left\{
\left\Vert f^{m}\right\Vert _{L_{1}(a,b)}+\left\vert f^{m}(x_{0})\right\vert
(b-a)\right\}  \tag{3}
\end{equation}%
In the same way, we can estimate $I_{14}(x_{0},\lambda ,m)$.From property b)

\begin{eqnarray}
I_{14}(x_{0},\lambda ,m) &\leq &\dsum\limits_{m=1}^{N}K_{\lambda
,m}(x_{0},x_{0}+\delta )\left\{ \dint\limits_{x_{0}+\delta }^{b}\left\vert
f^{m}(t)\right\vert dt+\dint\limits_{x_{0}+\delta }^{b}\left\vert
f^{m}(x_{0})\right\vert dt\right\}  \notag \\
&\leq &\dsum\limits_{m=1}^{N}K_{\lambda ,m}(x_{0},x_{0}+\delta )\left\{
\left\Vert f^{m}\right\Vert _{L_{1}(a,b)}+\left\vert f^{m}(x_{0})\right\vert
(b-a)\right\} .  \TCItag{4}
\end{eqnarray}%
On the other hand, Since $x_{0}$ is a lebesque point of $f,$ for every $%
\varepsilon >0,$ there exists a $\delta >0$ such that

\begin{equation}
\dint\limits_{x_{0}}^{x_{0}+h}\left\vert f(t)-f(x_{0})\right\vert
dt<\varepsilon h  \tag{5}
\end{equation}%
and%
\begin{equation}
\dint\limits_{x_{0}-h}^{x_{0}}\left\vert f(t)-f(x_{0})\right\vert
dt<\varepsilon h  \tag{6}
\end{equation}%
for all $0<h\leq \delta .$Now let's define a new function as follows,%
\begin{equation*}
F(t)=\dint\limits_{x_{0}}^{t}\left\vert f(u)-f(x_{0})\right\vert du.
\end{equation*}%
Then from (5), for $t-x_{0}\leq \delta $ we have 
\begin{equation*}
F(t)\leq \varepsilon (t-x_{0}).
\end{equation*}%
Also, since $f$ is bounded, there exists $M>0$ such that 
\begin{equation*}
\left\vert f^{m}(t)-f^{m}(x_{0})\right\vert \leq \leq \left\vert
f(t)-f(x_{0})\right\vert M
\end{equation*}%
is satisfied. Therefore, we can estimate $I_{13}(x_{0},\lambda ,m)$ as
follows. 
\begin{eqnarray*}
I_{13}(x_{0},\lambda ,m) &\leq
&M\dsum\limits_{m=1}^{N}\dint\limits_{x_{0}}^{x_{0}+\delta }\left\vert
f(t)-f(x_{0})\right\vert K_{\lambda ,m}(x_{0},t)dt \\
&\leq &M\dsum\limits_{m=1}^{N}\dint\limits_{x_{0}}^{x_{0}+\delta }K_{\lambda
,m}(x_{0},t)dF(t).
\end{eqnarray*}%
We apply integration by part, then we obtain the following result. 
\begin{equation*}
\left\vert I_{13}(x_{0},\lambda ,m)\right\vert \leq
M\dsum\limits_{m=1}^{N}\left\{ F(x_{0}+\delta ,x_{0})K_{\lambda
,m}(x_{0}+\delta ,x_{0})+\dint\limits_{x_{0}}^{x_{0}+\delta }F(t)d\left(
-K_{\lambda ,m}(x_{0},t)\right) \right\} .
\end{equation*}%
Since $K_{\lambda ,m}$ is decreasing on $\left[ x_{0},b\right] $, it is
clear that $-K_{\lambda ,m}$ is increasing.Hence its differential is
positive. Therefore, we can wirte%
\begin{equation*}
\left\vert I_{13}(x_{0},\lambda ,m)\right\vert \leq
M\dsum\limits_{m=1}^{N}\left\{ \varepsilon \delta K_{\lambda
,m}(x_{0}+\delta ,x_{0})+\varepsilon \dint\limits_{x}^{x+\delta
}(t-x_{0})d\left( -K_{\lambda ,m}(x_{0},t)\right) \right\} .
\end{equation*}%
Integration by parts again we have the following inequality%
\begin{eqnarray}
\left\vert I_{13}(x_{0},\lambda ,m)\right\vert &\leq &\varepsilon
M\dsum\limits_{m=1}^{N}\dint\limits_{x_{0}}^{x_{0}+\delta }K_{\lambda
,m}(x_{0},t)dt  \notag \\
&\leq &\varepsilon M\dsum\limits_{m=1}^{N}\dint\limits_{a}^{b}K_{\lambda
,m}(x_{0},t)dt.  \TCItag{7}
\end{eqnarray}%
Now, we can use similar method for evaluation $I_{12}(x_{0},\lambda ,m).$ Let%
\begin{equation*}
G(t)=\int_{t}^{x}\left\vert f(y)-f(x)\right\vert dy.
\end{equation*}%
Then, the statement 
\begin{equation*}
dG(t)=-\left\vert f(t)-f(x_{0})\right\vert dt.
\end{equation*}%
is satisfied. For $x_{0}-t\leq \delta $, by using (6),it can be written as
follows%
\begin{equation*}
G(t)\leq \varepsilon \left\vert x_{0}-t\right\vert
\end{equation*}%
Hence, we get%
\begin{equation*}
I_{12}(x_{0},\lambda ,m)\leq
M\dsum\limits_{m=1}^{N}\dint\limits_{x_{0}-\delta }^{x_{0}}\left\vert
f(t)-f(x_{0})\right\vert K_{\lambda ,m}(x_{0},t)dt.
\end{equation*}%
Then, we shall write\newline
\begin{equation*}
\left\vert I_{12}(x_{0},\lambda ,m)\right\vert \leq M\dsum\limits_{m=1}^{N} 
\left[ -\dint\limits_{x_{0}-\delta }^{x_{0}}K_{\lambda ,m}(x_{0},t)dG(t)%
\right] .
\end{equation*}%
By integration of parts, we have%
\begin{equation*}
\left\vert I_{12}(x,\lambda ,m)\right\vert \leq
M\dsum\limits_{m=1}^{N}\left\{ G(x-\delta K_{\lambda ,m}(x-\delta
,x)+\dint\limits_{x-\delta }^{x}G(t)d_{t}(K_{\lambda ,m}(x,t))\right\} .
\end{equation*}%
From (6), we obtain 
\begin{equation*}
\left\vert I_{12}(x_{0},\lambda ,m)\right\vert \leq
M\dsum\limits_{m=1}^{N}\left\{ \varepsilon \delta K_{\lambda
,m}(x_{0},x_{0}-\delta )+\varepsilon \dint\limits_{x_{0}-\delta
}^{x_{0}}(x_{0}-t)d_{t}(K_{\lambda ,m}(x_{0},t))\right\} .
\end{equation*}%
By using integration of parts again, we find 
\begin{equation}
\left\vert I_{12}(x_{0},\lambda ,m)\right\vert \leq \varepsilon
M\dsum\limits_{m=1}^{N}\dint\limits_{a}^{b}K_{\lambda ,m}(x_{0},t)dt. 
\tag{8}
\end{equation}%
Combined (7) and (8), we get%
\begin{equation}
\left\vert I_{12}(x_{0},\lambda ,m)\right\vert +\left\vert
I_{13}(x_{0},\lambda ,m)\right\vert \leq 2\varepsilon
M\dsum\limits_{m=1}^{N}\dint\limits_{a}^{b}K_{\lambda ,m}(x_{0},t)dt. 
\tag{9}
\end{equation}%
Finally,from (3), (4) and (9),the terms on right hand side of these
inequalitys tend to 0 as $\lambda \rightarrow \infty $. That's 
\begin{equation*}
\lim_{\lambda \rightarrow \infty }L_{\lambda
}(f,x_{0})=\dsum\limits_{m=1}^{N}C_{m}f^{m}(x_{0}).
\end{equation*}%
Thus, the proof is completed.
\end{proof}

In this theorem, specially it may be $a=-\infty $ and $b=\infty $. In this
case, we can give the following theorem.

\begin{theorem}
Let $f\in L_{1}(-\infty ,\infty )$ and $f$ is bounded. If non-negative the
kernel $K_{\lambda ,m}$ belongs to Class A and satisfies also the following
properties ,%
\begin{equation}
\lim_{\lambda \rightarrow \infty }\dint\limits_{-\infty }^{x-\delta
}K_{\lambda ,m}(t,x)dt=0  \tag{10}
\end{equation}%
and%
\begin{equation}
\lim_{\lambda \rightarrow \infty }\dint\limits_{x+\delta }^{\infty
}K_{\lambda ,m}(t,x)dt=0,  \tag{11}
\end{equation}%
then the statement 
\begin{equation*}
\lim_{\lambda \rightarrow \infty }L_{\lambda }(f,x)=f(x).
\end{equation*}%
is satisfied at almost every $x\in R$.

\begin{proof}
We can write%
\begin{eqnarray*}
\left\vert L_{\lambda }(f,x)-\dsum\limits_{m=1}^{N}C_{m}f^{m}(x)\right\vert
&\leq &\dsum\limits_{m=1}^{N}\dint\limits_{-\infty }^{\infty }\left\vert
f^{m}(t)-f^{m}(x)\right\vert K_{\lambda ,m}(x,t)dt \\
&&+\dsum\limits_{m=1}^{N}\left\vert f^{m}(x)\right\vert \left\vert
\dint\limits_{-\infty }^{\infty }K_{\lambda ,m}(x,t)dt-C_{m}\right\vert \\
&=&A_{1}(x,\lambda )+A_{2}(x,\lambda ).
\end{eqnarray*}%
It is clear that $A_{2}(x,\lambda )\rightarrow 0$ as $\lambda \rightarrow
\infty .$

For a fixed $\delta >0,$ we divide the integral $A_{1}(x,\lambda )$ of the
form 
\begin{eqnarray*}
A_{1}(x,\lambda ) &=&\dsum\limits_{m=1}^{N}\left[ \dint\limits_{-\infty
}^{x-\delta }+\dint\limits_{x-\delta }^{x}+\dint\limits_{x}^{x+\delta
}+\dint\limits_{x+\delta }^{\infty }\right] \left\vert
f^{m}(t)-f^{m}(x)\right\vert K_{\lambda ,m}(x,t)dt \\
&=&A_{11}(x,\lambda ,m)+A_{12}(x,\lambda ,m)+A_{13}(x,\lambda
,m)+A_{14}(x,\lambda ,m).
\end{eqnarray*}%
$A_{12}(x,\lambda ,m)$ and $A_{13}(x,\lambda ,m)$ integrals are calculated
as above the proof.For proof, it is sufficent to show that $A_{11}(x,\lambda
,m)$ and $A_{14}(x,\lambda ,m)$ tend to zero as $\lambda \rightarrow \infty
. $

Firstly, we consider $A_{11}(x,\lambda ,m).$ Since $f$ is bounded and by the
property b),this integration is written the form%
\begin{eqnarray*}
A_{11}(x,\lambda ,m) &\leq &M\dsum\limits_{m=1}^{N}\dint\limits_{-\infty
}^{x-\delta }\left\vert f(t)-f(x)\right\vert K_{\lambda ,m}(x,t)dt \\
&\leq &M\dsum\limits_{m=1}^{N}K_{\lambda ,m}(x,x-\delta )\left\{
\dint\limits_{-\infty }^{x-\delta }\left\vert f(t)\right\vert \right\}
+M\left\vert f(x)\right\vert \dsum\limits_{m=1}^{N}\dint\limits_{-\infty
}^{x-\delta }K_{\lambda ,m}(x,t)dt \\
&\leq &\left\Vert f\right\Vert _{L_{1}(-\infty ,\infty
)}M\dsum\limits_{m=1}^{N}K_{\lambda ,m}(x,x-\delta )+M\left\vert
f(x)\right\vert \dsum\limits_{m=1}^{N}\dint\limits_{-\infty }^{x-\delta
}K_{\lambda ,m}(x,t)dt.
\end{eqnarray*}%
In addition to, we obtain the inequality 
\begin{eqnarray*}
A_{14}(x,\lambda ,m) &\leq &M\dsum\limits_{m=1}^{N}\dint\limits_{x+\delta
}^{\infty }\left\vert f(t)-f(x)\right\vert K_{\lambda ,m}(x,t)dt \\
&\leq &\left\Vert f\right\Vert _{L_{1}(-\infty ,\infty
)}M\dsum\limits_{m=1}^{N}mK_{\lambda ,m}(x,x+\delta )+M\left\vert
f(x)\right\vert \dsum\limits_{m=1}^{N}M\dint\limits_{x+\delta }^{\infty
}K_{\lambda ,m}(x,t)dt.
\end{eqnarray*}%
According to the conditions d), (10) and (11),we find that $A_{11}(x,\lambda
,m)+A_{14}(x,\lambda ,m)\rightarrow 0$ as $\lambda \rightarrow \infty .$This
completes the proof.
\end{proof}
\end{theorem}

\bigskip 

This article is presented in " Young Scientist's International Baku Form,
May, 2013.

\end{document}